\documentclass{amsart}

\usepackage{amssymb}
\usepackage[all]{xy}
\usepackage{caption}
\usepackage[pdftex]{hyperref}
\usepackage{graphicx}
\usepackage{array}
\usepackage{physics}
\usepackage{bm}
\usepackage{xcolor}

\newtheorem{thm}{Theorem}[section]
\newtheorem*{theorem*}{Theorem}
\newtheorem{lem}[thm]{Lemma}
\newtheorem{prop}[thm]{Proposition}
\newtheorem{cor}[thm]{Corollary}

\newtheorem*{con*}{Conjecture}

\theoremstyle{remark}
\newtheorem{rem}[thm]{Remark}
\theoremstyle{definition}

\providecommand\sslash{\mathbin{/\mkern-5.5mu/}}

\newcommand{\sldos}{\SL(2,\mathbb{C})}

\newcommand{\CC}{{\mathbb C}}
\newcommand{\ZZ}{{\mathbb Z}}
\newcommand{\PP}{{\mathbb P}}
\newcommand{\RR}{{\mathbb R}}

\newcommand{\G}{{\Gamma}}
\newcommand{\cO}{{\mathcal O}}

\newcommand{\la}{\langle}
\newcommand{\ra}{\rangle}
\newcommand{\g}{\gamma}
\newcommand{\x}{\times}

\DeclareMathOperator{\Id}{id}
\DeclareMathOperator{\Pic}{Pic\,}           

\DeclareMathOperator{\Hom}{Hom\,}           
\DeclareMathOperator{\Stab}{Stab}

\DeclareMathOperator{\GL}{GL}
\DeclareMathOperator{\SU}{SU}
\DeclareMathOperator{\U}{U}
\let\S\relax
\newcommand{\S}{\mathrm{S}}

\DeclareMathOperator{\SL}{SL}

\DeclareMathOperator{\Sym}{Sym}    
\DeclareMathOperator{\SSym}{SSym}    
\DeclareMathOperator{\diag}{diag}

\title[Stratification of $\SU(r)$-character varieties of twisted Hopf links]{Stratification of $\SU(r)$-character varieties\\ of twisted Hopf links}

\author[A. Gonz\'alez-Prieto]{\'Angel Gonz\'alez-Prieto}

\address{Departamento de \'Algebra, Geometr\'ia y Topolog\'ia, Facultad de Ciencias Matem\'aticas, Universidad Complutense de Madrid, Plaza Ciencias 3, 28040 Madrid Spain.}
\address{Instituto de Ciencias Matem\'aticas (CSIC-UAM-UCM-UC3M), C.\ Nicolás Cabrera 13-15, 28049 Madrid Spain.}\email{angelgonzalezprieto@ucm.es}

\author[M. Logares]{Marina Logares}

\address{Departamento de \'Algebra, Geometr\'ia y Topolog\'ia, Facultad de Ciencias Matem\'aticas, Universidad Complutense de Madrid, Plaza Ciencias 3, 28040 Madrid Spain.}
\email{mlogares@ucm.es}

\author[J. Mart\'{\i}nez]{Javier Mart\'{\i}nez}

\address{Departamento de Matem\'atica Aplicada, Ciencia e Ingeniería de los Materiales y Tecnología Electrónica. E.S. Ciencias Experimentales y Tecnología, Universidad Rey Juan Carlos, C.\ Tulipán 0, 28933 Móstoles, Madrid Spain.}
\email{javier.martinezmar@urjc.es}

\author[V. Mu\~noz]{Vicente Mu\~noz}

\address{Instituto de Matem\'atica Interdisciplinar (IMI) and Departamento de \'Algebra, Geometr\'ia y Topolog\'ia, Facultad de Ciencias Matem\'aticas, Universidad Complutense de Madrid, Plaza Ciencias 3, 28040 Madrid Spain.}\email{vicente.munoz@ucm.es}


\begin{document}

\begin{abstract}
We describe the geometry of the character variety of representations of 
the fundamental
group of the complement of a Hopf link with $n$ twists, namely $\G_{n}=\la x,y \,| \, [x^n,y]=1 \ra$
into the group
$\SU(r)$. For arbitrary rank, we provide geometric descriptions of the loci of irreducible and totally reducible representations. In the case $r = 2$, we provide a complete geometric description of the character variety, proving that this $\SU(2)$-character variety is a deformation retract of the larger $\SL(2,\CC)$-character variety, as conjectured by Florentino and Lawton. In the case $r = 3$, we also describe different strata of the $\SU(3)$-character variety according to the semi-simple type of the representation.
\end{abstract}

\maketitle

\begin{center}
{\center \em Dedicated to Prof.\ Peter E.\ Newstead on the occasion of his 80th birthday.}
\end{center}

\section{Introduction}\label{sec:introduction}

Let $\G$ be a finitely generated group and $G$ a real or complex algebraic group. A
{representation} of $\G$ into $G$ is a group homomorphism $\rho: \G\to G$.
Consider a presentation $\G=\la \g_1,\ldots, \g_k \,|\, \{r_{\lambda}\} \ra$, 
where $\{r_\lambda\}$ is a finite generating set of relations of $\G$. The map $\rho$ is completely
determined by the $k$-tuple $(A_1,\ldots, A_k)=(\rho(\g_1),\ldots, \rho(\g_k))$
subject to the relations $r_\lambda(A_1,\ldots, A_k)=\Id$, for all $\lambda$.
In this way, the set of representations of $\G$ into $G$ is in bijection with the algebraic set
 $$
R(\G, G) =\, \{(A_1,\ldots, A_k) \in G^k \, | \, r_\lambda(A_1,\ldots, A_k)=\mathrm{id},  \forall \lambda \, \}\subset G^{k}\, .
$$

Two representations $\rho$ and $\rho'$ are said to be
equivalent if there exists $g \in G$ such that $\rho'(\gamma)=g^{-1} \rho(\gamma) g$,
for every $\gamma \in \G$. When $\rho$ is a faithful representation into  $G \subset \GL(V)$, the equivalence of representations means that $\rho$ and $\rho'$ are the same representation up to a $G$-change of basis of $V$. The moduli space of representations 
(or character variety) can be thus obtained as the GIT quotient
 $$
 X(\G, G) = R(\G, G) \hspace{-1pt}\sslash\hspace{-1pt} G\, .
 $$

An important instance happens when $G = \GL(r, \CC)$, for which we recover the classical notion of a linear representation 
as a $\G$-module structure on the vector space $\CC^r$.
It is worth noticing that, when $\G$ is a finite group, the vector space $\CC^r$ can be equipped with a $\G$-invariant hermitian metric. Hence, any representation $\rho: \G \to \GL(r, \CC)$ descends to an $\U(r)$-representation
	\[
\begin{displaystyle}
   \xymatrix
   {	\G \ar[rr]^{\rho} \ar@{-->}[rrd]_{\tilde{\rho}} && \GL(r, \CC) \\
    && \U(r) \ar@{^{(}->}[u]
      }
\end{displaystyle}   
\]
In other words, $X(\G, \GL(r, \CC)) \cong 
X(\G, \U(r))$. However, in the general case in which $\G$ is only finitely generated, such an invariant metric may not exist, so $X(\G, \U(r))$ is only a real subvariety of $X(\G, \GL(r, \CC))$. Similar considerations can be done in the case in which we fix the determinant of the representation, so we analyze the descending property of representations induced by the inclusion $\SU(r) \hookrightarrow \SL(r. \CC)$, which exhibits $X(\G, \SU(r))$ as a real subvariety of $X(\G, \SL(r, \CC))$.

For non-finite groups $\G$, the situation is different. In \cite{florentino2020flawed}, it is 
proved that, when $\G$ is a free product of nilpotent groups or a star-shaped RAAG (Right Angled Artin Group), 
the inclusion $X(\G, K) \hookrightarrow X(\G, G)$ is a deformation retract for any reductive group $G$ and its maximal compact subgroup $K \subset G$ (a property called `flawed'). On the contrary, for $\G = \pi_1(\Sigma_g)$, the fundamental group of a compact orientable surface $\Sigma_g$
of genus $g \geq 2$, 
this inclusion is never a homotopy equivalence when $G$ is reductive and non-abelian (it is said that $\G$ is a `flawless' group in the language of  \cite{florentino2020flawed}). 

The study of the homotopy type of these character varieties in the case of knot groups was initiated in \cite{Munoz}. For a knot $K\subset S^3$, the finitely presented group
$\G = \pi_1(S^3 - K)$ is an important invariant of the knot, called the
{knot group}. The first non-trivial family of knots studied in the literature has been the torus knot of type $(m,n)$, for which
the knot group is $\G=\la x,y \, | \, x^n=y^m\ra$.
For the torus knot group, \cite{Martinez-Munoz:2015} and \cite{Munoz} prove that such 
inclusion is a deformation retract in the case $\SU(2) \hookrightarrow \SL(2, \CC)$, and 
\cite{gonzalezmartinezmunoz2022} proves the same result in the case $\SU(3) \hookrightarrow \SL(3, \CC)$.
 
The aim of this work is to extend this program to a family of $3$-dimensional links. In analogy with the cases of knots, given a link $L \subset S^3$, the {link group} is the
fundamental group $\G=\pi_1(S^3-L)$ of its complement. For an algebraic group $G$,
we have the $G$-character variety of the link
$$
	X(L, G) = \Hom(\pi_1(S^3-L), G) \sslash G\, ,
$$
as before. 
Despite of the advances for representation varieties of knots, much less is known in the case of links. 
An obvious case is the character variety of trivial links, i.e.\ representations of the free group. Very recently, more complicated links have been 
studied, such as the twisted Alexander polynomial for the Borromean link in \cite{ChenYu}. 
In \cite{gonzalezmunoz2022hopf}, it is addressed the extension of the analysis of the geometry of character varieties for the 
the ``twisted'' Hopf link $H_n$, obtained by twisting a classical Hopf link with $2$ crossings to get $2n$ crossings, as depicted in Figure \ref{fig0}.

\begin{figure}[ht]
\begin{center}
\includegraphics[width=4cm]{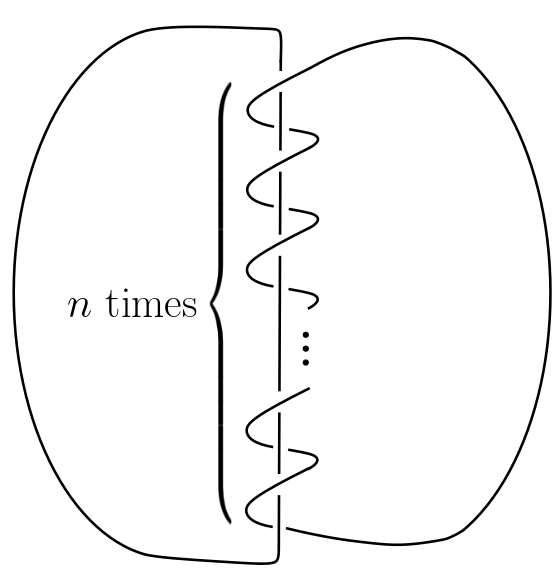}
\caption{\label{fig0} The twisted Hopf link of $n$ twists.}
\end{center}
\end{figure}

The fundamental group of the link complement of $H_n$ can be computed through a Wirtinger presentation \cite{gonzalezmunoz2022hopf},
giving rise to the group $\Gamma_n = \langle a, b \,|\, [a^n,b] = 1 \rangle$. Therefore, the associated $G$-character variety is
$$
	X(H_n, G) = \left\{(A,B) \in G^2 \,|\, [A^n,B] = 1\right\} \sslash G.
$$
In this sense, $X(H_n, G)$ should be understood as the variety counting ``supercommuting'' elements of $G$, generalizing the case $n=1$ of the usual Hopf link that corresponds to commuting elements.

In this setting, we shall show in this paper how the character varieties of twisted Hopf links can be stratified according to the type of its semisimple decomposition in the cases $G = \SU(r)$ and $G = \U(r)$, an approach that was initiated in \cite{SL4torusknot}. Moreover, we shall provide an explicit description of its irreducible locus (Section \ref{sec:irreducible}) and locus of representations that fully split into $1$-dimensional representations (Section \ref{sec:totally-reducible}). This will lead us to explore other auxiliary spaces such as symmetric spaces of tori. In Section \ref{sec:SU2character}, we will apply these techniques to describe the global geometry of the $\SU(2)$-character variety. Furthermore, in Section \ref{sec:homotopy-type} we will compare its topology with the one of the $\SL(2,\CC)$-character variety, variety that was previously studied in \cite{gonzalezmunoz2022hopf}. These analyses lead to the main result of this work.

\begin{theorem*}
The $\SU(2)$-character variety of a twisted Hopf link is a strong deformation retract of the corresponding $\SL(2, \CC)$-character variety. Moreover, they are homotopically equivalent to the wedge of $n$ copies of $2$-spheres. 
\end{theorem*}

Notice that this result provides the first non-trivial extension of the results of \cite{florentino2020flawed} to link groups. In the higher dimensional cases of $\U(2)$ and $\SU(3)$, we also provide a thorough description of the semisimple stratification arising in the corresponding character varieties. Even though each piece is fully described, the geometry involved is very complicated and a complete description of the intersection pattern that arises is unavailable. This prevents further cohomological calculations that are needed to characterize the homotopy type of these varieties.

Despite this essential obstruction, our results suggest that an alternative approach can be considered to address the problem of comparing the homotopy type of $\SU(r)$ and $\SL(r, \CC)$ character varieties. Indeed, the methods of this work show that, even though the retraction $X(H_n, \SL(2,\CC)) \to X(H_n, \SU(2))$ does not respect the semisimple strata, it does keep their closures invariant. In this way, even though we do not expect the existence of any retraction that shrinks each stratum locally, such deformation retract may be possible if we allow the homotopy to use neighbouring strata in the closure. With this idea, we pose the following conjecture.

\begin{con*}
For any $r \geq 1$, the $\SU(r)$-character variety of a twisted Hopf link is a strong deformation retract of the corresponding $\SL(r, \CC)$-character variety.
\end{con*}

Finally, we would like to point out that, even though throughout this paper all these spaces will be studied from a purely representation-theoretic approach, they have been traditionally related with moduli spaces of parabolic bundles. Indeed, the work of Furuta and Steer \cite{FurutaSteer} shows an algebro-geometric approach to the study of the topology of the moduli space of representations of the fundamental group of a Seifert fibered space into $\SU(r)$. 
This idea applies to the twisted Hopf links addressed in present work. Twisted Hopf links are $(2, 2n)$-torus links, that is $2$-bridge links which satisfy Burde and Murasugi theorem \cite[Theorem 1]{burde-murasugi}, so that their complement is a Seifert fibered $3$-sphere over an orbifold with one singular point of order $n$. 

In \cite[Section 6.7]{boden}, Boden shows that, given such Seifert fibered homology $3$-sphere $M$, there exists an orbifold surface $\Sigma$ such that the character variety $X(\pi_1(M), \SU(2))$ is isomorphic to $X^*(\pi_1^{\textrm{orb}}(\Sigma), \SU(2))$, the moduli space of irreducible representations of the orbifold fundamental group of $\Sigma$. This later space can be actually understood as an union of moduli spaces of stable parabolic bundles with appropriate parabolic data \cite{Mehta-Seshadri}. From this viewpoint, our work honours the work of Newstead in \cite{newstead67} in the study of the homotopy of moduli spaces of bundles. 

The aforementioned relation with moduli spaces of parabolic bundles is completed with an analogous relation between $\SL(2,\CC)$-character varieties of Seifert fibered homology $3$-spheres and moduli spaces of traceless parabolic Higgs bundles, thanks to the work of Nasatyr and Steer \cite{NasatyrSteer}. This would motivate us to seek a deformation retraction from the moduli space of traceless parabolic Higgs bundles to the moduli space of degree $0$ stable parabolic bundles. We do not expect it to happen, as it is in fact the case in the non-parabolic situation (see \cite{FlorentinoGothenNozad}). The reason may be the same: the moduli space of traceless parabolic Higgs bundles retracts to its well-known nilpotent cone, which contains the moduli space of degree $0$ stable parabolic bundles as one of its irreducible components, but it is indeed reducible. Nevertheless, we shall address this question in future work.

\subsection*{Acknowledgements}

The first author has been supported by the Madrid Government (Comunidad de Madrid – Spain) under the Multiannual Agreement with the Universidad Complutense de Madrid in the line Research Incentive for
Young PhDs, in the context of the V PRICIT (Regional Programme of Research and Technological Innovation) through the project PR27/21-029, COMPLEXFLUIDS research grant sponsored by the BBVA Foundation and by the Ministerio de Ciencia e Innovaci\'on Project PID2021-124440NB-I00 (Spain). The second author has been supported by the Ministerio de Ciencia e Innovaci\'on Project PID2021-124440NB-I00 (Spain),  Santander-UCM PR44/21-29924 research grant and by the National Science Foundation under grant No.\ DMS-1928930 while she  was in residence at the Simons Laufer Mathematical Sciences Institute (previously known as MSRI) Berkeley, California, during Fall 2022 Semester. The fourth 
author has been partially supported by Ministerio de Ciencia e Innovaci\'on Project PID2020-118452GB-I00 (Spain).

\section{Representation spaces and character varieties}
 We fix notation and outline here some of the properties that will be used throughout the paper. Let us consider a finitely generated group $\Gamma = \langle \gamma_1, \ldots, \gamma_k \,|\, r_\lambda(\gamma_1, \ldots, \gamma_k) \rangle$, where $r_\lambda(\gamma_1, \ldots, \gamma_k)$ is the set of relations satisfied by the generators $ \gamma_1, \ldots, \gamma_k$. Given an algebraic linear group $G$, the $G$-representation variety of $\Gamma$ is the space
 $$
 R(\Gamma,G)= \Hom(\Gamma,G) = \lbrace (g_1, \ldots, g_k) \in G^k \mid r_\lambda(g_1, \ldots, g_k)=\Id \rbrace.
 $$
 For our purposes, we shall focus on the classical matrix cases $G=\GL(r,\CC), \SL(r,\CC), \U(r)$ and $\SU(r)$. Given $\rho\in R(\Gamma,G)$, we will call $\rho$ \textit{reducible} if there exists some proper linear
subspace $W\subset \CC^r$ that is invariant for all $\g \in \G$, that is, such that
$\rho(\Gamma)(W)\subset W$; and we will say that $\rho$ is \textit{irreducible} otherwise. In the case that there exists a $\rho$-equivariant decomposition $\CC^r = W_1 \oplus \ldots \oplus W_s$ into irreducible components, $\rho$ is said to be \emph{semisimple}.

There is a natural action of the group $G$ on $ R(\Gamma,G)$ given by $g \cdot \rho=g\rho g^{-1}$ for $g \in G$ and $\gamma \in \Gamma$, namely the conjugation action. Two representations are said to be equivalent if they belong to the same $G$-orbit under the conjugation action. The moduli space of $G$-representations of $\Gamma$ is the GIT quotient (see \cite{newstead70} for an introduction to these algebraic quotients) with respect to this action, 
$$
X(\Gamma,G)=R(\Gamma,G)\sslash G.
$$
It is well known (see \cite{newstead70}) that $X(\Gamma,G)$ parametrizes representations up to an equivalence relation called S-equivalence: two representations $\rho$ and $\rho'$ are S-equivalent if the closures of their $G$-orbits intersect. It turns out that, for the classical groups $G=\GL(r,\CC), \SL(r,\CC)$, every representation is S-equivalent to a semisimple one. Furthermore, in the compact case $G = \U(r)$ and $\SU(r)$, every representation is semisimple (see \cite[Proposition 1]{gonzalezmartinezmunoz2022}). Hence, in all these cases, $X(\Gamma,G)$ parametrizes semisimple representations up to (classical) equivalence.

\subsection{Representations of the twisted Hopf link group}
We study here representations of the following group. For $n \geq 1$, let $H_n \subset S^3$ 
be the so-called twisted Hopf link, a generalization of the classical Hopf link with $2n$ crossings instead of $2$, as represented in Figure \ref{fig0}.

Using its Wirtinger presentation, in \cite[Proposition 3.1]{gonzalezmunoz2022hopf} it was shown that the fundamental group of the complement of the link $H_n$ is
$$
	\Gamma_n = \pi_1(S^3-H_n) = \langle a,b \mid \left[a^n, b\right] = 1\rangle,
$$
where $[x,y] = xyx^{-1}y^{-1}$ denotes the group commutator. In this way, the associated $G$-representation variety is
 $$
 R(\Gamma_n,G)= \Hom(\Gamma_n, G) = \lbrace (A,B)\in G^2 \mid \left[A^n,B \right]=\Id \rbrace.
 $$
 
In this paper we shall study the corresponding representation spaces and character varieties of $\Gamma_n$ for $G=\U(r)$ and $G=\SU(r)$. We will denote them as
$$
X_{r}=X(\Gamma_{n},\U(r)),\qquad  \S X_{r}=X(\Gamma_{n},\SU(r)).
$$

\subsection{Stratification of representations}

Let us fix $r \geq 1$. Recall that a partition $\pi$ of $r$ is a decomposition
$$
r=a_1r_1+\ldots+a_sr_s\, ,
$$
where $r_1>r_2>\ldots>r_s>0$ and $a_i\geq 0$. 
We shall denote it by $\pi=(r_1,\overset{(a_1)}{\ldots},r_1,\ldots,r_s,\overset{(a_s)}{\ldots},r_s)$. Associated to this partition, we can consider equivalence classes of semisimple representations $\rho: \Gamma_n \to \U(r) $ of type $\pi$, described as:
\begin{equation}
\rho=\bigoplus_{t=1}^{s}\bigoplus_{l=1}^{a_l} \rho_{tl}, \quad \rho_{tl}: \Gamma_n \rightarrow \U(r_t). \label{subrep}
\end{equation}
We denote the set of such representations by $X^{\pi}_r \subset X_r = X(\Gamma_{n},\U(r))$.
Analogously, for $\SU(r)$ we set $\S X^\pi_r = X^\pi_r \cap \S X_r$, which is made of those representations in (\ref{subrep}) with $\prod_{t,l} \det(\rho_{tl})=1$.

As a consequence, the character varieties $X_{r}$ and $SX_{r}$ decompose into locally closed subvarieties indexed by the set $\Pi_r$ of all partitions of $r$:
$$
X_{r}=\bigsqcup_{\pi\in\Pi_r} {X}^{ \pi}_r, \qquad \S X_{r}=\bigsqcup_{\pi\in\Pi_r} \S X^{\pi}_r.
$$

Among all possible partitions, the following two cases will be important: $X^{\textrm{TR}}_r = X_r^{\pi_1}$, the set of \emph{totally reducible} representations, that corresponds to $\pi_1=(1,\overset{(r)}{\ldots},1)$, and $X^{\ast}_r = X^{\pi_0}_r$, which will denote the set of \emph{irreducible} representations, where $\pi_0=(r)$. Analogous notation will be used for the $\SU(r)$ counterparts.

Finally, given a (real) algebraic variety $X$, let us consider the symmetric product of $r$ copies of $X$, $\Sym^r(X) = X^r/ S_r$, where the symmetric group $S_r$ acts by permutation of the factors. With this symmetric product, we have the following simple characterization of semisimple representations up to equivalence.

\begin{lem}\label{cor:reducible-structure}
For any partition $\pi=(r_1,\overset{(a_1)}{\ldots},r_1,\ldots,r_s,\overset{(a_s)}{\ldots},r_s) \in \Pi_r$, we have an isomorphism
$$
	{X}^\pi_r \cong \prod_{i=1}^s \Sym^{a_i}(X_{r_i}^*).
$$
\end{lem}

\subsection{Irreducible representations}\label{sec:irreducible}

First, let us observe the following useful property of irreducible representations of the twisted Hopf link group.
\begin{lem}\label{lem:schur}
Let $\rho = (A,B): \Gamma_n \to G$ be an irreducible representation. Then $A^n$ is a multiple of the identity. \label{prop:irreduciblecondition}
\end{lem}

\begin{proof}
The relation in $\Gamma_{n}$ says that $A^n$ is an equivariant linear map from $\mathbb{C}^{r}$ to $\mathbb{C}^{r}$. Hence, by 
Schur's lemma, $A^n$ is a multiple of the identity.
\end{proof}

Consider the natural projection map
$$
	\omega: X_r^* \to \Sym^r(S^1).
$$ 
that assigns, to each irreducible representation $(A, B) \in X_r^{\ast}$, the collection of eigenvalues of the matrix $A$.
To study this map, we stratify $\Sym^r(S^1)$ according to the number of repeated eigenvalues. To be precise, given a partition $\sigma = (r_1,\overset{(a_1)}{\ldots},r_1,\ldots,r_s,\overset{(a_s)}{\ldots},r_s) \in \Pi_r$, let us denote by $\Sym^r_{\sigma}(S^1)$ the collection of sets $\{\lambda_1, \ldots, \lambda_r\}$ such that there are $a_i$ groups of $r_i$ equal eigenvalues (i.e.\ $a_i$ eigenvalues with multiplicity $r_i$). In particular, $\sigma_0 = (1,\overset{(r)}{\ldots},1)$ corresponds to the case of different eigenvalues. In this way, given $\sigma \in \Pi_r$, we set
$$
	{X}_{\sigma}^\ast = \omega^{-1}(\Sym^r_{\sigma}(S^1)).
$$
These are equivalence classes of representations $(A,B)$ such that $A$ has repeated eigenvalues given by $ \sigma$. Hence, we get a further stratification of the irreducible locus ${X}_{r}^*$ according to the number of repeated eigenvalues
\begin{equation}\label{eq:strata-type}
	{X}_{r}^* = \bigsqcup_{\sigma \in \Pi_r} {X}_{\sigma}^\ast.
\end{equation}
Notice that some of the strata ${X}_{\sigma}^\ast$ may be empty in this decomposition. The maximal component of ${X}_{r}^\ast$ corresponds to the partition $\sigma_0 = (1,\overset{(r)}{\ldots},1)$ of non-repeated eigenvalues. Provided that $n\geq r$, this stratum is non-empty: once we diagonalize $A$ the irreducibility of the representation ensures that we can choose  two non-coincident basis for the eigenvectors of $A$ and $B$, which provides an element of the stratum.

Given $\sigma \in \Pi_r$, let us pick $\{\lambda_1, \ldots, \lambda_r\} \in \Sym^r_\sigma(S^1)$, and let us denote by $\Stab(\sigma)$ the $\U(r)$-stabilizer of the diagonal matrix $A_\sigma = \textrm{diag}(\lambda_1, \ldots, \lambda_r)$. Notice that $\Stab(\sigma)$ does not depend on the particular choice of eigenvalues $\{\lambda_1, \ldots, \lambda_r\} \in \Sym^r_\sigma(S^1)$. Indeed, if $\sigma = (r_1,\overset{(a_1)}{\ldots},r_1,\ldots,r_s,\overset{(a_s)}{\ldots},r_s)$, we have the explicit description
$$
	\Stab(\sigma) = \prod_{i = 1}^s \SU(r_i)^{a_i}.
$$

\subsubsection{$\SU(r)$-irreducible representations}\label{sec:irreducible-su}
If we specialize to the case of $\SU(r)$, directly from Lemma \ref{lem:schur} we get the following.

\begin{cor}
Let  $\rho = (A,B): \Gamma_n \to \SU(r)$ be an irreducible representation. Then $A^n=\xi\Id$, where $\xi \in \bm{\mu}_r$ is an $r$-th root of unity. In particular, $A$ admits finitely many eigenvalues.
\end{cor}

This result implies that the map
$$
	\omega: \S X_r^* \to \SSym^r(S^1)
$$ 
has finite image, where $\SSym^r(S^1) \subset \Sym^r(S^1)$ is the collection of sets $\{\lambda_1, \ldots, \lambda_r\} \in \Sym^r(S^1)$ such that $\lambda_1 \cdots \lambda_r = 1$. The image is precisely those eigenvalues additionally satisfying $\lambda_1^n=\lambda_2^n=\ldots=\lambda_r^n \in \bm{\mu}_r$. Analogously to the previous setting, we also have stratifications
\begin{equation}\label{eq:stratification-eigen-rep}
	\SSym^r(S^1) = \bigsqcup_{\sigma \in \Pi_r} \SSym^r_\sigma(S^1), \qquad \S X_r^\ast = \bigsqcup_{\sigma \in \Pi_r} \S X_{\sigma}^*,
\end{equation}
 according to the number of repeated eigenvalues.

Given $\sigma \in \Pi_r$, let us denote by $N_\sigma$ the number of elements $\{\lambda_1, \ldots, \lambda_r\} \in \SSym^r_\sigma(S^1)$ such that $\lambda_1^n = \lambda_2^n = \ldots = \lambda_r^n = \xi \in \bm{\mu}_r$ (and $\lambda_1 \cdots \lambda_r=1$). Fixed $\xi\in \bm{\mu}_r$, the number of possibilities was computed in \cite[Corollary 6.7]{SL4torusknot} for the case $\textrm{gcd}(n,r) = 1$. Hence, accounting for the $r$ choices of $\xi$, we get that
$$
	N_\sigma =  \frac{r}{n} \begin{pmatrix}n \\ a_1, a_2, \ldots, a_s\end{pmatrix},
$$
at least for $n$ and $r$ coprime (conjecturally, for any $n$ and $r$). Here, we have used the multinomial coefficients
$$
	\begin{pmatrix}n \\ a_1, a_2, \ldots, a_s\end{pmatrix} = \frac{n!}{a_1! a_2! \ldots a_s! (n - a_1 - \ldots - a_s)!}.
$$

\begin{lem}\label{lem:irreducible}
Let $\sigma \in \Pi_r$. We have that $SX_\sigma^*$ is a disjoint union of $N_\sigma$ copies of a subspace
$$
	\S \mathbf{F}_\sigma \subset \SU(r)/\Stab(\sigma),
$$
where the action of $\Stab(\sigma)$ on $\SU(r)$ is by conjugation.
\end{lem}

\begin{proof}
Let $(A,B) \in SX_\sigma^*$. Since $A$ is diagonalizable, $(A,B)$ has a representative with the matrix 
$A = \textrm{diag}(\lambda_1, \ldots, \lambda_r)$ for some $\{\lambda_1, \ldots, \lambda_r\} \in \SSym^r_\sigma(S^1)$. This form is unique up to the action of $\Stab(A) = \Stab(\sigma)$ by conjugation on $B$, so it defines a point in $\SU(r)/\Stab(\sigma)$, as claimed.
\end{proof}

\begin{rem}\label{rmk:SFtau}
The subspace $\S \mathbf{F}_\sigma$ corresponding to $\S X_\sigma^*$ can be characterized explicitly. Let $\S\mathbf{F}_\sigma^0$ by the preimage of $\S \mathbf{F}_\sigma$ under the projection map $\SU(r) \to \SU(r)/\Stab(\sigma)$. Then $\S \mathbf{F}_\sigma^0$ is given by the classes of matrices $B \in \SU(r)$ such that $B$ has no  proper invariant subspace in common with $A = \textrm{diag}(\lambda_1, \ldots, \lambda_r)$. Notice that the invariant subspaces of $A$ are certain subspaces generated by the canonical basis vectors of $\CC^r$, depending on the partition $\sigma$. 
\end{rem}

\subsubsection{$\SU(r)$-irreducible representations with distinct eigenvalues}\label{sec:distinct-eigenvalues}

In this section, we shall further study the case of the partition $\sigma_0 = (1,\overset{(r)}{\ldots},1)$ corresponding to non-repeated eigenvalues. In this setting, $\S \mathbf{F}_{\sigma_0}^0$ is the collection of orthonormal bases $\{b_1, \ldots, b_r\}$ of $\CC^r$ with volume $1$ such that $\langle b_{i_1}, \ldots, b_{i_k}\rangle \neq \langle e_{i_1}, \dots, e_{i_k} \rangle$ for any proper subset $\{i_1, \ldots, i_k\} \subset \{1, \ldots, r\}$. Moreover $\Stab(\sigma_0) = (S^1)^r$ with the action $(\lambda_1, \ldots, \lambda_r) \cdot (b_{ij}) = (\lambda_i\lambda_j^{-1} b_{ij})$ for $(\lambda_1, \ldots, \lambda_r) \in (S^1)^r$ and $B = (b_{ij}) \in \SU(r)$.

Let us first analyze the space $\SU(r) / (S^1)^r$ with the previous action. Recall that $\SU(r)$ can be understood as a fiber bundle over $S^{2r-1}$ (the choice of the first basis vector) whose fiber is a fiber bundle over $S^{2r-3}$ (the choice of the second basis vector), whose fiber is again a fiber bundle over $S^{2r-5}$ and so on so forth until the last basis vector, that in principle belongs to $S^1$ but is actually fixed by the volume condition.

However, the action of $(S^1)^r$ makes things more involved. Given a matrix $B = (b_1 \,|\, \ldots \, |\, b_r)$, where $b_j = (b_{1j}, \ldots, b_{rj})$ are its column vectors, the action allows us to arrange $b_{i1} \in \RR_{\geq 0}$ for all $i \geq 2$. Such representant is unique, so we get a projection map
$$
	\varphi: \SU(r)/(S^1)^r \longrightarrow B_r
$$
 onto the ``coarse orthant'' (c.f. \cite[Theorem 10]{gonzalezmartinezmunoz2022})
$$
	B_r = \{(z, x_2, \ldots, x_r) \in \CC \times \RR_{\geq 0}^{r-1} \,|\, |z|^2 + x_2^2 + \ldots + x_r^2 = 1\} \subset S^{r} \subset S^{2r-1}.
$$
The fiber of $\varphi$ over $b=b_1 = (z, x_2, \ldots, x_r)$ is determined by the number of ways in which $b$ can be completed to an orthonormal basis of volume $1$. On the interior of $B_r$ the fiber $\varphi^{-1}(b)$ is an iterated bundle of spheres, but when $b$ belongs to the boundary of $B_r$, there remains a residual action of $G_b = \{(\lambda_1, \ldots, \lambda_r) \in (S^1)^r \,|\, \lambda_1 = \lambda_i \textrm{ if } x_i = 0 \textrm{ for } 2 \leq i \leq r\}$ acting on $S^{2r-3}\times \ldots \times S^3$.

If we restrict our attention to $\S \mathbf{F}_{\sigma_0} = \S \mathbf{F}_{\sigma_0}^0/(S^1)^r$, then we must remove those bases that lead to a reducible representation. Such bases only occur at the boundary of $B_r$, so the previous description also holds for $\S \mathbf{F}_{\sigma_0}$ on the interior of $B_r$.

\subsubsection{$\U(r)$-irreducible representations}\label{sec:ur-irreducible}

The case of $\U(r)$-representations is analogous to the previous one, and we get an eigenvalue map
$$
	{\omega}: X_r^* \to \Sym^r(S^1).
$$ 
However, in this case the map no longer has finite image since, for a configuration of eigenvalues $\{\lambda_1, \ldots, \lambda_r\} \in \Sym^r(S^1)$ in the image, the element $\xi=\lambda_1^n=\lambda_2^n=\ldots=\lambda_r^n$ is an arbitrary point of $S^1$. We start the analysis of these representations with the simplest case.
\begin{lem}\label{lem:reducible-structure}
We have ${X}_1^* \cong S^1 \times S^1$.

\begin{proof}
Since $\U(1) = S^1$ is an abelian group, any pair $(A,B) \in \U(1) \times \U(1) = S^1 \times S^1$ leads to a representation for the twisted Hopf link. Moreover, since they are $1$-dimensional representations, they are automatically irreducible. 
\end{proof}
\end{lem}

The higher rank case $r > 1$ is more involved. As a first step, we have an analogous result to Lemma \ref{lem:irreducible}.

\begin{prop}\label{prop:ur-irreducible}
Fix $\sigma \in \Pi_r$. There is a Zariski locally trivial fiber bundle
$$
	\omega_\sigma: {X}_\sigma^* \to \Sym^r_{\sigma}(S^1)
$$
with fiber ${\mathbf{F}}_\sigma \subset \U(r)/\Stab(\sigma)$ and base the collection of eigenvalues $\{\lambda_1, \ldots, \lambda_r\} \in \Sym^r_\sigma(S^1)$ such that $\lambda_1^n=\lambda_2^n=\ldots=\lambda_r^n$.
\end{prop}

\begin{proof}
The spectrum map $\sigma: \U(r) \to \Sym^r_{\sigma}(S^1)$ is locally trivial in the Zariski topology. Indeed, on a trivializing open set $U \subset \Sym^r_{\sigma}(S^1)$ we may conjugate each representation $(A,B)$ in $\omega_\sigma^{-1}(U)$ so that $A = \diag(\lambda_1, \ldots, \lambda_r)$. In this form, $B \in \U(r)$ is uniquely determined up to the conjugacy action of $\Stab(A) = \Stab(\sigma)$.
\end{proof}

\begin{rem}\label{rmk:Ftau}
Again, the subspace ${\mathbf{F}}_\sigma$ can be characterized explicitly in full analogy with $\S \mathbf{F}_\sigma$ as in Remark \ref{rmk:SFtau}, with the only difference that now $B \in \U(r)$ instead of $B\in \SU(r)$.
\end{rem}

To look closer at this fibration, let us consider the variety $\hat{X}_r^* = X_r^* \times_{\Sym^r(S^1)} (S^1)^r$ given as the pullback
	\[
\begin{displaystyle}
   \xymatrix
   {
   \hat{X}_r^* \ar[r] \ar[d]& X_r^* \ar[d]^{{\omega}_\sigma} \\
   (S^1)^r \ar[r] & \Sym^r(S^1) 
      }
\end{displaystyle}   
\]
where $(S^1)^r \to \Sym^r(S^1)$ is the quotient map. Analogously, if we want to restrict the number of coincident eigenvalues, we set $\hat{X}_\sigma^* =  {X}_\sigma^* \times_{\Sym^r_\sigma(S^1)} (S^1)^r$ for $\sigma \in \Pi_r$, and we denote by $\Sigma_\sigma \subset (S^1)^r$ the image $\hat{X}_\sigma^* \to (S^1)^r$.

\begin{lem}
The fibration $\hat{X}_\sigma^* \to \Sigma_\sigma$ is trivial, so we have an isomorphism
\begin{equation}\label{eq:fibration-ur-irr}
	\hat{X}_\sigma^* \cong {\mathbf{F}}_\sigma \times \Sigma_\sigma.
\end{equation}
\end{lem}

\begin{proof}
The elements of $\hat{X}_\sigma^*$ are tuples $(A, B, \lambda_1, \ldots, \lambda_r) \in \U(r)^2 \times \Sigma_\sigma$ such that $(A,B)$ is an irreducible representation and the spectrum of $A$ is $\{\lambda_1, \ldots, \lambda_r\}$. Any such representation has a canonical representative with $A = \textrm{diag}(\lambda_1, \ldots, \lambda_r)$ and thus $B \in {\mathbf{F}}_\sigma$, leading to the desired isomorphism. 
\end{proof}

In the case that there exists at least one eigenvalue with multiplicity one, we can describe the fiber of the fibration (\ref{eq:fibration-ur-irr}) in terms of the one for $\SU(r)$.
\begin{prop}
If $\sigma  = (r_1,\overset{(a_1)}{\ldots},r_1,\ldots,r_s,\overset{(a_s)}{\ldots},r_s) \in \Pi_r$ is a partition with $r_1 = 1$ (i.e.\, there exists at least one eigenvalue with multiplicity one), then we have an isomorphism
$$
	{\mathbf{F}}_\sigma \cong \S{\mathbf{F}}_\sigma \times S^1.
$$
\end{prop}

\begin{proof} Fix eigenvalues $\{\lambda_1, \ldots, \lambda_r\} \in \Sym^r_\sigma(S^1)$. To simplify the notation, we shall suppose that the first eigenvalue $\lambda_1$ is simple. Given an element $(A,B, \lambda_1, \ldots, \lambda_r) \in {\mathbf{F}}_\sigma$, the representation can be conjugated to one of the form
$$
	(A, B) = (\textrm{diag}(\lambda_1, \ldots, \lambda_r), (b_1 \,|\, b_2 \,|\, \ldots \,|\, b_r)),
$$
where $b_i$ are the column vectors of $B$. Moreover, since $\lambda_1$ is a simple eigenvalue, the vector $b_1$ is well defined up to re-scaling by action of $S^1$, the stabilizer of the first eigenvector of $A$. In this way, if set $\mu = \det(B)$, we have that $B' = (\mu^{-1}b_1 \,|\, b_2 \,|\, \ldots \,|\, b_r) \in \SU(r)/\Stab(\sigma)$. Hence, the map ${\mathbf{F}}_\sigma \ni B \mapsto (B', \mu) \in \S {\mathbf{F}}_\sigma \times S^1$ provides the desired isomorphism.
\end{proof}

Regarding the base of the fibration in (\ref{eq:fibration-ur-irr}), we can easily describe it in combinatorial terms as follows.

\begin{lem}
Given $\sigma = (r_1,\overset{(a_1)}{\ldots},r_1,\ldots,r_s,\overset{(a_s)}{\ldots},r_s) \in \Pi_r $, let $N = a_1 + \ldots + a_s$ be the number of different eigenvalues. Then, we have that
$$
	\Sigma_\sigma \cong S^1 \times \Delta^{N-1}_{\bm{\mu}_n},
$$
where $\Delta^{N-1}_{\bm{\mu}_n}$ is the collection of tuples $(\epsilon_2, \ldots, \epsilon_{N})$ with $\epsilon_i \in \bm{\mu}_n^* = \bm{\mu}_n - \{1\}$ such that $\epsilon_i \neq \epsilon_j$ for $i \neq j$.
\end{lem}

\begin{proof}
Let $(\lambda_1, \ldots, \lambda_N)$ be the ordered different eigenvalues of an element of $\Sigma_\sigma$. Since $\lambda_1^n = \lambda_2^n = \ldots = \lambda_N^n$, for $i \geq 2$ we can uniquely write $\lambda_i = \lambda_1 \epsilon_i$ for some $\epsilon_i \in \bm{\mu}_n^*$ with all the roots $\epsilon_2, \ldots, \epsilon_N$ different. Hence, the map $(\lambda_1, \ldots, \lambda_N) \mapsto (\lambda_1, \epsilon_2, \ldots, \epsilon_N)$ yields the isomorphism.
\end{proof}

Putting together all these pieces, we finally get an explicit description of ${X}_\sigma^*$, at least in the case where there exists a simple eigenvalue.

\begin{cor}\label{cor:irreducible-u-simplecase}
Let $\sigma = (r_1,\overset{(a_1)}{\ldots},r_1,\ldots,r_s,\overset{(a_s)}{\ldots},r_s)  \in \Pi_r$ be a partition with at least one eigenvalue with multiplicity one, and let $N = a_1 + \ldots + a_s$ be the number of different eigenvalues. Then we have an isomorphism
$$
	{X}_\sigma^* = \left(\Sigma_\sigma \times  {\mathbf{F}}_\sigma \right)/S_n 
	= \left(S^1 \times \Delta^{N-1}_{\bm{\mu}_n} \times  \S{\mathbf{F}}_\sigma \times S^1 \right)/S_n,
$$
where the action of the symmetric group $S_n$ on $S^1 \times \Delta^{N-1}_{\bm{\mu}_n}$ is given by permutation of eigenvalues and on ${\mathbf{F}}_\sigma = \S{\mathbf{F}}_\sigma \times S^1$ by permutation of columns.
\end{cor}

\subsection{Totally reducible representations}\label{sec:totally-reducible}

In this section, we analyze the representations corresponding to the partition $\pi_1=(1,\overset{(r)}{\ldots},1)$, that is, the spaces of totally reducible representations $X_r^{\textrm{TR}}$ and $\S X_r^{\textrm{TR}}$ for the $\U(r)$ and $\SU(r)$ cases, respectively. By Lemma \ref{lem:reducible-structure} we have that for $G = \U(r)$
$$
	X_r^{\textrm{TR}} = \Sym^r(X_1^*) = \Sym^r(S^1 \times S^1).
$$
Analogously, for $G = \SU(r)$ we get
$$
	\S X_r^{\textrm{TR}} = {\SSym}{}^r(S^1 \times S^1),
$$
where $\SSym^r(S^1 \times S^1)$ is the subset of $\Sym^r(S^1 \times S^1)$ of sets $\{(\lambda_1, \mu_1), \ldots, (\lambda_r, \mu_r)\}$ such that $\lambda_1 \cdots \lambda_r = \mu_1 \cdots \mu_r = 1$.

Our first result in this direction is that the former space is a fibration with fiber the later one.

\begin{prop}\label{prop:symtr}
We have a Zariski locally trivial fibration
$$
	\varpi_\sigma: \Sym^r_\sigma(S^1 \times S^1) \to S^1 \times S^1, \quad 
	\varpi_\sigma(\{(\lambda_1, \mu_1), \ldots, (\lambda_r, \mu_r)\}) = (\lambda_1 \cdots \lambda_r, \mu_1 \cdots \mu_r),
$$
whose fiber is isomorphic to $\SSym{}^r(S^1 \times S^1)$.
\end{prop}

\begin{proof}
The base of the fibration is parametrized by $(t,s)=(\lambda_1 \cdots \lambda_r,\mu_1 \cdots \mu_r) \in S^1\times S^1$, and it is clear that $\SSym{}^r(S^1\times S^1)=\phi^{-1}(1,1)$.
The total space can be trivialized over any $(t,s)\in S^1\times S^1$ via the map $$\phi: \varpi_\sigma^{-1}(t,s) \x (S^1-\{-t\})\times (S^1-\{-s\}) \to \varpi_\sigma^{-1}(S^1-\{-t\}\times S^1-\{-s\}),$$
given by $\phi(\{(\lambda_1, \mu_1), \ldots, (\lambda_r, \mu_r)\},te^{i\theta},se^{i\alpha}) = \{(e^{i\theta/r}\lambda_1, e^{i\alpha/r}\mu_1), \ldots, (e^{i\theta/r}\lambda_r, e^{i\alpha/r}\mu_r)\}$, for $\theta, \alpha \in (-\pi,\pi)$.
\end{proof}

Now, let us consider the projection onto the first component
$$
	\varpi: \Sym^r(S^1 \times S^1) \to \Sym^r(S^1).
$$
As in Section \ref{sec:irreducible-su}, given a partition $\sigma \in \Pi_r$, we shall focus on the subset $\Sym^r_{\sigma}(S^1) \subset \Sym^r(S^1)$ of configurations of points with coincident entries given by $\sigma$, and we set $\Sym^r_\sigma(S^1 \times S^1) = \varpi^{-1}(\Sym^r_\sigma (S^1))$. We analogously consider the determinant $1$ case and denoted by $\SSym^r_\sigma(S^1 \times S^1)$ the corresponding space.

Furthermore, given a partition $\sigma = (r_1,\overset{(a_1)}{\ldots},r_1,\ldots,r_s,\overset{(a_s)}{\ldots},r_s) \in \Pi_r$ of $r$, let $S_\sigma = S_{r_1}^{a_1} \times \cdots \times S_{r_s}^{a_s} < S_r$ be the subgroup of permutations of type $\sigma$. Then, we define
$$
	{\Sym}^{\sigma}(S^1) = (S^1)^r/S_\sigma,
$$
and analogously $\SSym^{\sigma}(S^1) \subset {\Sym}^{\sigma}(S^1)$ for the configurations $\{\mu_1, \ldots, \mu_r\}$ such that $\mu_1 \cdots \mu_r = 1$. Observe in particular that $\Sym^r(S^1) = {\Sym}^{(r)}(S^1)$ and $\SSym^r(S^1) = {\SSym}^{(r)}(S^1)$.

\begin{prop}\label{prop:strati-widesym}
Fixed $\sigma = (r_1,\overset{(a_1)}{\ldots},r_1,\ldots,r_s,\overset{(a_s)}{\ldots},r_s) \in \Pi_r$, the fibration
\begin{equation}\label{eq:strati-widesym}
	 \SSym^r_\sigma(S^1 \times S^1) \to \SSym^r_\sigma(S^1).
\end{equation}
is Zariski locally trivial with fiber $\SSym^{\sigma}(S^1)$.

\begin{proof}
Fix $\{\lambda_1, \ldots, \lambda_r\} \in \SSym^r_\sigma(S^1)$ corresponding to the first factor. The algorithm described in \cite[Proposition 5]{gonzalezmartinezmunoz2022} provides a unique way of choosing a logarithm cut $\alpha$ such that we have a natural ordering $\lambda_{i_1} = e^{2\pi i(\alpha + \theta_1)}, \lambda_{i_2} = e^{2\pi i(\alpha + \theta_2)}, \ldots, \lambda_{i_r} = e^{2\pi i(\alpha + \theta_r)}$ with $\theta_1 \leq \theta_2 \leq \ldots \leq \theta_r$. Hence, the $r$ points in the second factor of the fiber $\varpi^{-1}(\{\lambda_1, \ldots, \lambda_r\})$ are uniquely ordered except in the cases where we have coincident values in the first factor. Hence, to remove this ordering, we need to quotient by the action of $S_\sigma$, as claimed.

Observe that, once a configuration of repeated values fixed, the choice of the logarithm cut $\alpha$ varies algebraically with the points $\{\lambda_1, \ldots, \lambda_r\}$. Indeed, $\alpha$ corresponds to the argument of one of them, and the chosen value only changes when it surpasses another one. Hence, the construction above is locally trivial in the Zariski topology, as stated.
\end{proof}
\end{prop}

\begin{rem}\label{rmk:standard-conf}
Notice that the base space $\SSym^r(S^1)$ has already been studied in \cite[Proposition 5]{gonzalezmartinezmunoz2022} and turns out to be isomorphic to the $(r-1)$-simplex
$$
	\Delta_{r-1} = \left\{(u_1, \ldots, u_{r}) \in \RR_{\geq 0}^r \,\left| \, \sum_{i=1}^r u_i = 0 \right.\right\}.
$$
In particular, all the base spaces $ \SSym^r_\sigma(S^1)$ of Proposition \ref{prop:strati-widesym} are simplices and thus contractible, so the fiber bundle (\ref{eq:strati-widesym}) is topologically trivial.
\end{rem}

\begin{rem}\label{rmk:sym-divisors}
Using some basic algebraic geometry, it is possible to give an equivalent description of $\Sym^r(S^1 \times S^1)$. Let us fix an almost-complex structure on $S^1 \times S^1$, which turns it into an elliptic curve $\Sigma$. In this setting, the elements of $\Sym^r(S^1 \times S^1)$ correspond exactly to effective divisors on $\Sigma$ of degree $r$. Under this interpretation, if $\Pic_r(\Sigma)$ is the collection of (holomorphic) line bundles on $\Sigma$ of degree $r$, we can consider the map
\begin{equation}\label{eq:fib-sym}
	\Sym^r(S^1 \times S^1) \to \Pic_r(\Sigma), \qquad D \mapsto \cO(D)
\end{equation}
that sends each divisor $D$ into its induced line bundle $\cO(D)$. Note that, since $\Sigma$ is an elliptic curve, $\Pic_r(\Sigma) \cong S^1 \times S^1$ topologically. The fiber of this map at $\cO(D)$ is precisely the collection of effective line bundles linearly equivalent to $D$, which is the complex projective space $\PP(H^0(\cO(D)))$. Using that the canonical divisor $K$ of $\Sigma$ has degree zero, we get that $\cO(K-D)$ has no global holomorphic sections and thus, by the Riemann-Roch theorem, we get $h^0(\cO(D)) = r$. Hence, we get that (\ref{eq:fib-sym}) is a bundle over $S^1 \times S^1$ with fiber $\CC\PP^{r-1}$.
\end{rem}

\subsubsection{The spaces $\Sym^2(S^1 \times S^1)$ and $\SSym^2(S^1 \times S^1)$}\label{sec:sym-2}

In this section, we shall study more closely configurations of two points in $S^1 \times S^1$. In the first place, $\SSym^2(S^1 \times S^1)$ corresponds to collections of sets $\{(\lambda, \lambda^{-1}), 
(\mu, \mu^{-1})\}$. We stratify this space according to the possible repetitions of partitions $\sigma \in \Pi_2$, as in Proposition \ref{prop:strati-widesym}. Recall from Remark \ref{rmk:standard-conf} that $\SSym^2(S^1) = \overline{\mathbf{I}}=[-1,1]$, the closed interval. Its strata are $\SSym^2_{(1,1)}(S^1) = \mathbf{I} = (-1,1)$, the open interval (different eigenvalues), and $\SSym^2_{(2)}(S^1) = \{-1,1\}$, the endpoints (repeated eigenvalues).
\begin{enumerate}
	\item For $\sigma = (1,1)$ we have different values in the first factor. In this case,we get an open cylinder since
	$$
		\SSym^2_{(1,1)}(S^1 \times S^1) = \SSym^2_{(1,1)}(S^1) \times \SSym^{(1,1)}(S^1) = \mathbf{I}  \times S^1,
	$$
	where we have used that $\SSym^{(1,1)}(S^1) = S^1$ and that the fiber bundle is topologically trivial since the base space $\mathbf{I}$ is contractible.
	\item For $\sigma = (2)$ we have repeated eigenvalues in the first factor. Hence, we get
	$$
		\SSym^2_{(2)}(S^1 \times S^1) = \SSym^2_{(2)}(S^1) \times \SSym^2(S^1) = \{\pm 1\} \times \overline{\mathbf{I}}.
	$$
\end{enumerate}
Hence, globally $\SSym^2(S^1 \times S^1)$ is the cylinder $\overline{\mathbf{I}} \times S^1$ with a gluing in the two boundaries $\{-1,1\} \times S^1$ that collapses each circle into an interval. Hence, $\SSym^2(S^1 \times S^1)$ is topologically (but not smoothly) a $2$-sphere: it is a pillowcase with four orbifold points (see Figure \ref{fig:su-char-var}).

\begin{rem}\label{rem:alternative-SSym2}
An alternative way of understanding $\SSym^2(S^1 \times S^1)$ is as the quotient $\SSym^2(S^1 \times S^1) = (S^1 \times S^1) / \ZZ_2$, where the $\mathbb{Z}_2$ action is given by $(\lambda,\mu) \mapsto (\lambda^{-1},\mu^{-1})$. Regarding the torus as a quotient of $[0,1]^2$, the action becomes $(s,t)\sim (1-s,1-t), s,t\in[0,1]$, whose quotient space is homeomorphic to the 2-sphere.
\end{rem}

With respect to the space $\Sym^2(S^1 \times S^1)$, notice that by Proposition \ref{prop:symtr} we have a description as $S^2$-bundle
$$
	S^2 \to \Sym^2(S^1 \times S^1) \to S^1 \times S^1.
$$
Observe that this perfectly fits with the description in Remark \ref{rmk:sym-divisors} since $\CC\PP^1 \cong S^2$.
The monodromy of this fiber bundle is given by 
$(\lambda,\mu) \mapsto (e^{i\pi}\lambda, e^{i\pi}\mu)=(-\lambda,-\mu)$. In coordinates
$(s,t)$, it is given by $(s,t)\mapsto (s+1/2,t+1/2)$, which is orientation preserving, and interchanges 
the four orbifold points of $S^2$ in pairs.

\subsubsection{The spaces $\Sym^3(S^1 \times S^1)$ and $\SSym^3(S^1 \times S^1)$}\label{sec:sym-3}

Now, let us look at the $3$-fold symmetric product $\SSym^3(S^1 \times S^1)$. We stratify this space according to the number of repeated eigenvalues in the first factor:
	\begin{enumerate}
		\item $\sigma = (1,1,1)$ (three different eigenvalues). In this case, we have
		$$
			\SSym^3_{(1,1,1)}(S^1 \times S^1) = \SSym^3_{(1,1,1)}(S^1) \times \SSym^{(1,1,1)}(S^1) = \mathbf{T} \times S^1 \times S^1,
		$$
		where $\mathbf{T} =  \SSym^3_{(1,1,1)}(S^1)$ is the open $2$-dimensional triangle as described by Remark \ref{rmk:standard-conf}.
		\item $\sigma = (1,2)$ (two coincident eigenvalues). This corresponds to the interior of the edges of the triangle $\mathbf{T}$ above. In this case, we have 
		$$
			\SSym^3_{(1,2)}(S^1 \times S^1) = \SSym^3_{(1,2)}(S^1) \times \SSym^{(1,2)}(S^1).
		$$
		Observe that, solving for the last eigenvalue, we have $\SSym^{(1,2)}(S^1) = \Sym^2(S^1)$, which is a M\"obius band (see \cite[Corollary 6]{gonzalezmartinezmunoz2022}). Hence, over each point of the interior of the edges $\mathbf{T}$ we find a M\"obius band.
		\item $\sigma = (1,3)$ (three equal eigenvalues). This corresponds to the vertices of the triangle $\mathbf{T}$. Now, 
		$$
			\SSym^3_{(3)}(S^1 \times S^1) = \SSym^3_{(3)}(S^1) \times \SSym^{3}(S^1),
		$$
		where this later space is a closed triangle. Hence, over each vertex of $\mathbf{T}$, we attach a triangle. Each of these triangles glues with the M\"obius bands of the incoming edges through their boundaries.
	\end{enumerate}

With respect to the total space $\Sym^2(S^1 \times S^1)$, notice that by Remark \ref{rmk:sym-divisors} we have a description as a fiber bundle
$$
	\CC\PP^2 \to \Sym^3(S^1 \times S^1) \to S^1 \times S^1.
$$

\section{$\SU(2)$-character variety} \label{sec:SU2character}

For $n=2$, there are only two partitions,  $\pi_{0}=(2)$ and $\pi_{1}=(1,1)$, that correspond to the set of irreducible representations and the set of totally reducible representations, respectively.

\begin{enumerate}
\item The partition $\pi_0=(2)$ yields the space of irreducible representations $\S X_2^*$. Notice that the only non-empty stratum for the stratification (\ref{eq:stratification-eigen-rep}) of $\S X_2^*$ corresponds to $\sigma=(1,1)$ since, in the case $\sigma=(2)$, the matrix $A$ of a representation $\rho = (A,B)$ is a multiple of the identity, so every representation is reducible.

Following the notation of Section \ref{sec:irreducible}, the partition $\sigma=(1,1)$ has 
$s=1$, $r_1 = 1$ and $a_1=2$. Hence, according to Lemma \ref{lem:irreducible}, we have that $\S X_{(1,1)}^*$ has as many components $\S \mathbf{F}_{(1,1)}$ as
$$
	N_{(1,1)} = \frac{2}{n}\begin{pmatrix} n \\ 2 \end{pmatrix} = \frac{2}{n} \frac{n!}{2! \, (n-2)!} = n-1.
$$
The stabilizer for this stratum is $\Stab((1,1)) = S^1 \times S^1$. To determine the space $\S \mathbf{F}_{(1,1)} \subset \SU(2)/(S^1 \times S^1)$, we observe that, by the discussion of Section \ref{sec:distinct-eigenvalues}, we have a fibration
$$
	\varphi: \S \mathbf{F}_{(1,1)}  \to B_2 \subset S^2.
$$
On the interior of $B_2$, this map is an isomorphism. On the boundary, we get no irreducible representations since the first vector of the basis must be colinear with $e_1$. Hence, we get that $\S \mathbf{F}_{(1,1)}$ is homeomorphic to $\mathbf{D} = \textrm{Int}(B_2)$, which is a $2$-dimensional open disc.

\begin{rem}
An equivalent way of describing $\S \mathbf{F}_{(1,1)}$ is the following. With the notation of Remark \ref{rmk:SFtau}, we have that $\S \mathbf{F}_{(1,1)}^0$ is the collection of orthonormal bases of $\CC^2$ whose vectors are not proportional to any of the vectors of the canonical basis. Since these bases are determined by the first vector, which can be seen as an element  $(a,b)\in S^3\subset \CC^2$, we get that $\S \mathbf{F}_{(1,1)}^0 = S^3 -S^1 $ is the $3$-sphere minus the $1$-sphere of equation $ b=0$. Now, the action of $S^1$ on $(a,b) \in S^3 \subset \CC^2$ is $\lambda \cdot (a,b) = (a, \lambda b)$, so the quotient is the weighted projective space
$$
	\S \mathbf{F}_{(1,1)}^0/(S^1 \times S^1) = (S^3-S^1)/S^1 = \CC\PP^1_{(0,1)} - \{[1:0]\}.
$$
This latter space can be understood as the set of $(a,b) \in S^3\subset \CC^2$ such that $b \in \RR_{> 0}$ so 
$\CC\PP^1_{(0,1)} - \{[1:0]\}$ is the upper hemisphere of $S^2$ and thus homeomorphic to a disc.
\end{rem}

\item The partition $\pi_1=(1,1)$ corresponds to the set of totally reducible representations $\S X_2^{\textrm{TR}}$. By the results of Section \ref{sec:totally-reducible}, we have $\S X_2^{\textrm{TR}} = \SSym^2(S^1 \times S^1)$, which is topologically (but not smoothly) a $2$-sphere.

\end{enumerate}

Summarising from the discussion above, we have that
 $$
 \S X_{2} = \S X^{(1,1)}_2 \sqcup \S X^{(2)}_2=\S X^{\textrm{TR}}_2 \sqcup \S X^*_{(1,1)}\, .
 $$ 
The global picture of $\S X_2$ is as follows (see also Figure \ref{fig:su-char-var}). First, the space $\S X^{\textrm{TR}}_2$ is a $2$-dimensional sphere. To it, we attach $n-1$ open discs $\mathbf{D}$ corresponding the the components of $\S X^*_{(1,1)}$, each of them attached to $\S X^{\textrm{TR}}_2$ through their boundary $S^1\subset \overline{\mathbf{D}}$. 
These boundaries correspond to reducible representations given by $A\sim \textrm{diag}(\lambda,\lambda^{-1})$ and $B\sim \textrm{diag}(\mu,\mu^{-1})$, where $\lambda\in \bm{\mu}_{2n}^{+}$, $\mu \in S^1$. They all inject into $\S X^{\textrm{TR}}_2$ as $n-1$ disjoint circles. The space is homotopically equivalent to the wedge of $n$ copies of $2$-spheres.

\begin{figure}[ht]
\begin{center}
\includegraphics[width=7cm]{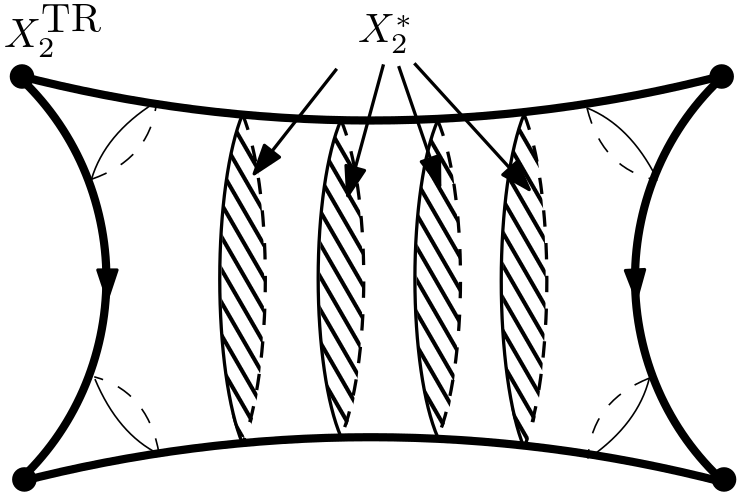}
\caption{The global picture of $X_2 = X(\G_n, \SU(2))$.}\label{fig:su-char-var}
\end{center}
\end{figure}

  \section{Homotopy type of the $\sldos$-character variety of twisted Hopf links}\label{sec:homotopy-type}

In this section, we study the homotopy type of the $\sldos$-character variety of the twisted Hopf link. Recall that this variety is
$$
	X(\G_n,\sldos) = \Hom(\Gamma_n, \sldos) /\hspace{-1mm}/ \sldos 
	= \{(A, B) \in \sldos \,|\, [A^n, B] = \Id \} /\hspace{-1mm}/ \sldos .
$$
This is a complex algebraic variety whose motive has been previously studied in \cite{gonzalezmunoz2022hopf}. The aim of this section is to prove that the natural inclusion map $\S X_2 = X(\G_n,\SU(2)) \hookrightarrow X(\G_n,\sldos)$ is a deformation retract. Notice that $\SU(2) \hookrightarrow \sldos$ is the maximal compact subgroup and that $\sldos$ is the complexification $\sldos = \SU(2)_\CC$ of $\SU(2)$.

As in the $\SU(2)$-case, we can decompose the $\sldos$-character variety into its reducible and irreducible locus
$$
	X(\G_n,\sldos) = X^{\textrm{TR}}(\G_n,\sldos)  \sqcup X^*(\G_n,\sldos).
$$
In sharp contrast with the $\SU(2)$-case, not every reducible representation is semisimple. However, it turns out that, for the conjugacy action of $\sldos$, every representation has a semisimple representation in the closure of its orbit. This implies that every reducible representation is identified in the GIT quotient with a semisimple one, so the reducible locus of the character variety $X^{\textrm{TR}}(\G_n,\sldos)$ parametrizes totally reducible representations up to equivalence, which justifies the assumed notation.

Regarding the irreducible locus $X^*(\G_n,\sldos)$, recall from the results of \cite{gonzalezmunoz2022hopf} that it has $n-1$ components, corresponding to the possible eigenvalues of the matrix $A$ of a representation $(A,B) \in X^*(\G_n,\sldos)$. Fixed eigenvalues $\lambda, \lambda^{-1} \in \CC^*$ for $A$, we can choose a representative of the form
$$
	(A,B) \sim \left(\begin{pmatrix}\lambda & 0 \\ 0& \lambda^{-1}\end{pmatrix}, \begin{pmatrix} a & c \\ b& d\end{pmatrix}\right),
$$
with $bc \neq 0$ since $(A,B)$ is irreducible. This representative is not unique, since we get a residual action of $\CC^*$ on $B \in \sldos$ given by
$$
	\lambda \cdot \begin{pmatrix} a & c \\ b& d\end{pmatrix} = \begin{pmatrix} a & \lambda^{-1}c \\ \lambda b& d\end{pmatrix}, \; \lambda \in \CC^{\ast}.
$$
As a consequence, we get that each of these components of $X^*(\G_n,\sldos)$ is isomorphic to the space 
$$
	\S\mathbf{F}_\CC := \left\{(a,b,c,d) \in \CC^4 \,|\, ad-bc = 1, bc \neq 0\right\} \sslash \CC^*.
$$
The invariant coordinates for this $\CC^*$-action are $p := bc$ and $a,d$, so we finally get the  description
$$
	\S\mathbf{F}_\CC = \left\{(a,d,p) \in \CC^3 \,|\, ad-p = 1, p \neq 0\right\} \cong \CC^2 - H,
$$
where $H$ is the hyperbola $H = \{ad = 1\} \cong \CC^*$ and the later isomorphism is $(a,d,p) \mapsto (a,d)$. Inside this variety, the corresponding component $\S \mathbf{F}$ of the $\SU(2)$-character variety (see Section \ref{sec:SU2character}) is the open disc
$$
	\S\mathbf{F} = \S\mathbf{F}_{(1,1)} = \left\{(a,\bar{a}, |a|^2-1) \in \S\mathbf{F}_\CC \,\,\,|\,\,\, |a| < 1\right\} \cong \mathbf{D}.
$$
Indeed, given $(a, \bar{a}, |a|^2 -1 ) \in \S\mathbf{F}$, the corresponding matrix $B$ is
$$
	B = \begin{pmatrix} a & - \sqrt{1 - |a|^2}\\  \sqrt{1 - |a|^2}& \bar{a}\end{pmatrix}.
$$

\begin{rem}\label{rem:homotopy-original}
The spaces $\S\mathbf{F}$ and $\S\mathbf{F}_\CC$ are not homotopic. A straightforward computation with compactly supported cohomology shows that $H_c^k(\S\mathbf{F}_\CC) = H_c^k(\CC^2 - \CC^*)  = \ZZ$ for $k = 2,3,4$ and $H_c^k(\S\mathbf{F}_\CC) = 0$ otherwise. Hence, the Betti numbers of $\S\mathbf{F}_\CC$ are $b_k(\S\mathbf{F}_\CC) = 1$ for $k = 0,1,2$ and $b_k(\S\mathbf{F}_\CC) = 0$ for $k \geq 3$, showing that $\S\mathbf{F}_\CC$ cannot be contractible, as it is $\S\mathbf{F}$.
\end{rem}

Now, in $\S\mathbf{F}_\CC$, the set $p=0$ corresponds exactly to totally reducible representations in the boundary of the irreducible locus $X^*(\G_n,\sldos)$. Hence, the closure
$$
	\overline{\S\mathbf{F}}_\CC := \left\{(a,d,p) \in \CC^3 \,|\, ad-p = 1\right\} \cong \CC^2,
$$
with coordinates $(a,d) \in \CC^2$, is the collection of representations $(A,B) \in X^*(\G_n,\sldos)$ where $A$ has fixed eigenvalues $\lambda, \lambda^{-1}$. The corresponding closure of the component of the $\SU(2)$-character variety is thus the closed disc
$$
	\overline{\S\mathbf{F}} = \left\{(a,\bar{a}, |a|^2-1) \in \overline{\S\mathbf{F}}_\CC \,\,\,|\,\,\, |a| \leq 1\right\} \cong \overline{\mathbf{D}},
$$
with coordinates $(a,\bar{a}) \in \CC^2$.

\begin{lem}\label{lem:first-retraction}
Let $\overline{\S\mathbf{F}}_\CC' = \left\{(a,d) \in \overline{\S\mathbf{F}}_\CC \,|\, |a| = |d| \right\}$. There exists a smooth homotopy
$$
	H_t: \overline{\S\mathbf{F}}_\CC \to \overline{\S\mathbf{F}}_\CC,
$$
for  $t \in [0,1]$ with $H_0 = \Id_{\overline{\S\mathbf{F}}_\CC}$, $H_1(\overline{\S\mathbf{F}}_\CC) \subseteq \overline{\S\mathbf{F}}_\CC'$, and $H_t|_{\overline{\S\mathbf{F}}_\CC'} = \Id_{\overline{\S\mathbf{F}}_\CC'}$ for all $t$. Under this homotopy, the space $\overline{\S\mathbf{F}}_\CC - \S\mathbf{F}_\CC$ of $\sldos$-reducible representations remains invariant and is rescaled into the space $\overline{\S\mathbf{F}} - \S\mathbf{F}$ of $\SU(2)$-reducible representations. 
\end{lem}

\begin{proof}
Given $(a,d) \in \overline{\S\mathbf{F}}_\CC = \CC^2$, let us use polar coordinates $a = re^{i\alpha}$ and $d = s e^{i\beta}$, with $r,s \in \RR_{\geq 0}$ and $\alpha, \beta \in [0, 2\pi)$. Consider the auxiliary continuous homotopies $h^1, h^2: \RR_{\geq 0}^2 \times [0,1] \to \RR_{\geq 0}$ given by
$$
	h_t^1(r,s) = \left\{
	\begin{array}{ll}(1-t)r+t \sqrt{rs} & \textrm{for } r \geq s, \\
	\frac{rs}{(1-t)s+t \sqrt{rs}} & \textrm{for } s > r,
	\end{array}
	\right. \quad h_t^2(r,s) = \left\{
	\begin{array}{ll} \frac{rs}{(1-t)r+t \sqrt{rs}} & \textrm{for } r > s, \\
	 (1-t)s+t \sqrt{rs} & \textrm{for } s \geq r.
	\end{array}
	\right.
$$
Observe that for all $r, s \geq 0$ we have $h_0^1(r,s) =r$, $h_0^2(r,s) =s$, $h_1^1(r,s) = h_1^2(r,s) = \sqrt{rs}$. Moreover, we have $h_t^1(r,r) = h_t^2(r,r) = r$ and $h_t^1(r,s) \cdot h_t^2(r,s) = rs$ for all $t \in [0,1]$.

In this setting, we consider the homotopy $H: \overline{\S\mathbf{F}}_\CC \times [0,1] \to \overline{\S\mathbf{F}}_\CC$ given by
$$
	H_t(a,d) = \left( h_t^1(r,s) e^{i\alpha}, h_t^2(r,s) e^{i\beta}\right).
$$ 
Notice that this map makes sense even for $r = 0$ or $s=0$ since, in these cases, $h_t^1(0,s) = 0$ and $h_t^2(r,0)=0$, respectively.
This map satisfies the following properties:
\begin{enumerate}
	\item For $t = 0$ we get, for any $(a,d) \in \overline{\S\mathbf{F}}_\CC$,
$$
	H_0(a,d) =  \left(re^{i\alpha}, se^{i\beta}\right) = (a,d).
$$ 
Thus, $H_0 = \Id_{\overline{\S\mathbf{F}}_\CC}$.
	\item For $t = 1$ we get, for any $(a,d) \in \overline{\S\mathbf{F}}_\CC$,
$$
	H_1(a,d) = \left(\sqrt{rs} \,e^{i\alpha}, \sqrt{rs}\, e^{-i\beta}\right).
$$ 
In particular, $H_1(a,d) \in \overline{\S\mathbf{F}}_\CC'$ for all $a,d \in \overline{\S\mathbf{F}}_\CC$.

	\item For any point of the form $(a = re^{i\alpha}, d = r e^{i\beta}) \in \overline{\S\mathbf{F}}_\CC'$, we get
$$
	H_t(a,d) =  \left(re^{i\alpha}, r e^{i\beta}\right) = (a, d)
$$ 
for all $t \in [0,1]$. In particular, $H_t|_{\overline{\S\mathbf{F}}_\CC'} = \Id_{\overline{\S\mathbf{F}}_\CC'}$.
	\item For a point of the form $(a = re^{i\alpha}, a^{-1} = r^{-1}  e^{-i\alpha}) \in \overline{\S\mathbf{F}}_\CC - {\S\mathbf{F}_\CC}$ (equivalently, for $p=0$), we have
\begin{align*}
	H_t(a,a^{-1}) &=  \left( h_t^1\left(r,r^{-1}\right) e^{i\alpha}, h_t^2\left(r,r^{-1}\right) e^{-i\alpha}\right)
\end{align*}
for all $t \in [0,1]$. Hence, since $h_t^1(r,r^{-1})  \cdot h_t^2(r,r^{-1}) =1$ for all $t$, we have that $H_t(a,a^{-1}) \in \overline{\S\mathbf{F}}_\CC- {\S\mathbf{F}_\CC}$ for all $t$ or, in other words, $\overline{\S\mathbf{F}}_\CC- {\S\mathbf{F}_\CC}$ is invariant. The homotopy there is a rescaling.
\end{enumerate}

Therefore, property (1) shows that this map defines a homotopy equivalence between $H_0 = \Id_{\overline{\S\mathbf{F}}_\CC}$ and $H_1$. Moreover, by properties (2) and (3), $H$ defines a strong deformation retraction onto $\overline{\S\mathbf{F}}_\CC'$. By (4), this homotopy has the desired property on $\overline{\S\mathbf{F}}_\CC- {\S\mathbf{F}_\CC}$.
\end{proof}

\begin{rem}
Since $\overline{\S\mathbf{F}} \subset \overline{\S\mathbf{F}}_\CC'$, this space remains fixed under the previous retraction.
\end{rem}

Figure \ref{fig:retraction} shows the geometric interpretation of the homotopy of the proof of Lemma \ref{lem:first-retraction} in the quadrant $(r,s)$: the hyperbolas $rs = k$, for constant $k$, are retracted onto the plane $\{s = r\}$ through the natural $\RR_{\geq 0}$-action $\lambda \cdot (r,s) = (\lambda^{-1}r, \lambda s)$ for $\lambda \in \RR_{\geq 0}$. In particular, the hyperbola $\overline{\S\mathbf{F}}_\CC - \S\mathbf{F}_\CC = \{rs=1\}$ remains invariant, and the plane $\overline{\S \mathbf{F}}_\CC' = \{r = s\}$ is fixed.

\begin{figure}[ht]
\begin{center}
\includegraphics[width=5cm]{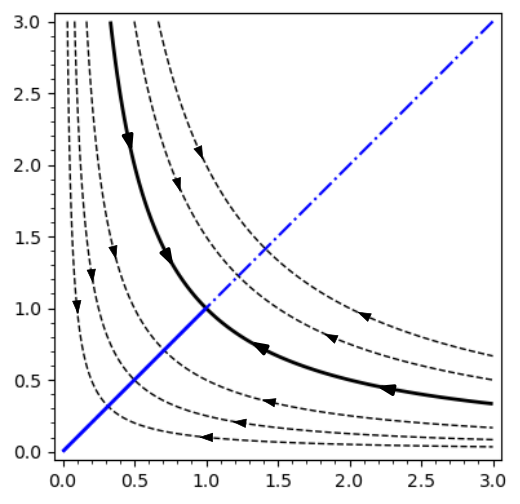}
\caption{\label{fig:retraction} Deformation retract of $\overline{\S\mathbf{F}}_\CC$ onto $\overline{\S\mathbf{F}}_\CC'$ in the $(r,s)$-plane.}
\end{center}
\end{figure}

\begin{lem}\label{lem:angle}
Let $\overline{\S\mathbf{F}}_\CC'' = \left\{(a,\bar{a}) \in \overline{\S\mathbf{F}}_\CC \right\}$. There exists a smooth homotopy
$$
	H_t: \overline{\S\mathbf{F}}_\CC' \to \overline{\S\mathbf{F}}_\CC,
$$
for $t \in [0,1]$ with $H_0 = \Id_{\overline{\S\mathbf{F}}_\CC'}$, $H_1(\overline{\S\mathbf{F}}_\CC') \subseteq \overline{\S\mathbf{F}}_\CC''$, and $H_t|_{\overline{\S\mathbf{F}}_\CC''} = \Id_{\overline{\S\mathbf{F}}_\CC''}$ for all $t$. 
\end{lem}

\begin{proof}
Let us write $a = re^{i\alpha}$ and $d = r e^{i\beta}$ with $r,s \in \RR_{\geq 0}$ and $\alpha, \beta \in [0, 2\pi)$. We consider the homotopy $H: \overline{\S\mathbf{F}}_\CC' \times [0,1] \to \overline{\S\mathbf{F}}_\CC$ given by
$$
	H_t(a,d) = \left\{\begin{array}{ll}\left( r e^{i\alpha}, r \left((1-t)e^{i\beta} + t e^{-i\alpha} \right)\right) & \textrm{if } a, d \neq 0, \\
	(0,0) & \textrm{if } a = d = 0.
	\end{array}\right. 
$$
Observe that, since $|a| = |d|$ in $ \overline{\S\mathbf{F}}_\CC'$, then any of them vanishes if and only if both vanish.

We obviously have $H_0 = \Id_{\overline{\S\mathbf{F}}_\CC'}$ and $H_1(a,d) = (a, \bar{a}) \in \overline{\S\mathbf{F}}_\CC'' $. Moreover, the points of the form $(a = re^{i\alpha}, \bar{a} = re^{-i\alpha})$ remain fixed, so we directly get that $H$ is the required homotopy.
\end{proof}

\begin{rem}
Since $\overline{\S\mathbf{F}} \subset \overline{\S\mathbf{F}}_\CC''$, this space remains fixed under the previous retraction.
\end{rem}

\begin{rem}
Notice that, throughout the homotopy of Lemma \ref{lem:angle}, the norm of the second component varies, so the target space of this homotopy is the whole $\overline{\S\mathbf{F}}_\CC$ and not the subspace $\overline{\S\mathbf{F}}_\CC'$. This is related to the following fact. Let $\Delta = \{(a,a)\} \subseteq \overline{\S\mathbf{F}}_\CC'$ be the diagonal. The space $\overline{\S\mathbf{F}}_\CC' - \Delta$ can be easily retracted to $\overline{\S\mathbf{F}}_\CC''$ by joining $(a,d)$ with $(a, \bar{a})$ through the shortest arc joining $d$ and $\bar{a}$. However, on $\Delta = \CC$, retracting it into $\overline{\S\mathbf{F}}_\CC''$ is the same as homotopying the conjugation map $f: \CC \to \CC$, $f(a) = \bar{a}$, into the identity map. If we restrict to $f|_{S^1}: S^1 \to S^1$, this is impossible, but it is possible as maps on $\CC$ if we allow the homotopy to pass through $0 \in \CC$ with the simple linear homotopy.
\end{rem}

Now, observe that $\S\mathbf{F}_\CC'' = \left\{(a,\bar{a}) \right\}$ radially retracts onto the disc $\overline{\S\mathbf{F}} = \S\mathbf{F}_\CC'' \cap \{|a| \leq 1\}$. Therefore, we have proven the following result.

\begin{prop}\label{prop:irredretract}
Each of the components $\overline{\S\mathbf{F}}_\CC$ of the closure of the irreducible locus of the $\sldos$-character variety strongly retracts onto the corresponding component $\overline{\S\mathbf{F}}$ of the closure of the irreducible locus of the $\SU(2)$-character variety. On the set of totally reducible representations $\overline{\S\mathbf{F}}_\CC - {\S\mathbf{F}_\CC}$, the homotopy is just linear rescaling onto $\overline{\S\mathbf{F}} - {\S\mathbf{F}}$. 
\end{prop}

\begin{rem}
At the light of Remark \ref{rem:homotopy-original}, Proposition \ref{prop:irredretract} has a clear interpretation: the non-trivial elements of $H_1(\S\mathbf{F}_\CC)$ and $H_2(\S\mathbf{F}_\CC)$ are annihilated when we glue back the hyperbola $\{ad = 1\}$ of totally reducible representations. Hence, only after taking the closures, the inclusion $\overline{\S\mathbf{F}} \hookrightarrow \overline{\S\mathbf{F}}_\CC$ becomes a homotopy equivalence.
\end{rem}

On the other hand, we can easily prove that the natural inclusion $\S X_2^{\textrm{TR}} \hookrightarrow X^{\textrm{TR}}(\G_n,\sldos)$ is a deformation retract.

\begin{prop}\label{prop:redretract}
The totally reducible locus $\S X_2^{\textup{TR}}$ of the $\SU(2)$-character variety is a strong deformation retract of the reducible locus of the $\sldos$-character variety $X^{\textrm{\textup{TR}}}(\G_n,\sldos)$ through a linear rescaling.
\end{prop}

\begin{proof}
By \cite{gonzalezmunoz2022hopf}, the reducible locus is given by pairs $(\lambda,\mu) \in (\CC^{\ast})^2$ quotiented by the $\ZZ_2$-action that identifies $(\lambda,\mu)\sim (\lambda^{-1},\mu^{-1})$. The radial deformation retract of $(\CC^\ast)^2$ onto $(S^1)^2$ descends to a deformation retraction of $(\CC^{\ast})^2/\ZZ_2$ onto $\S X_2^{\textrm{TR}}\cong \SSym^2(S^1 \times S^1)$ (c.f.\ Remark \ref{rem:alternative-SSym2}), since it commutes with the $\mathbb{Z}_2$-action. 
\end{proof}

Now, observe that on the common locus of totally reducible representations that lie in the closure of irreducible ones, the homotopies of Proposition \ref{prop:irredretract} and \ref{prop:redretract} coincide. Hence, we can glue them together to give rise to a global homotopy $H: X(\G_n,\sldos) \times [0,1] \to X(\G_n,\SL(2,\CC))$ such that $H_0 = \Id_{X(\G_n,\sldos)}$, $H_t|_{X(\G_n,\SU(2))} = \Id_{X(\G_n,\SU(2))}$ and $H_1(X(\G_n,\SL(2,\CC)))\subseteq X(\G_n,\SU(2))$. Therefore, we have proven the following result.

\begin{thm}\label{thm:retract}
The $\SU(2)$-character variety $\S X_2 = X(\G_n,\SU(2))$ of a twisted Hopf link is a strong deformation retract of the $\sldos$-character variety $X(\G_n,\sldos)$.
\end{thm}

\section{$\U(2)$ and $\SU(3)$ character varieties}

In this section, using the results of Section \ref{sec:irreducible} and \ref{sec:totally-reducible}, we shall describe the stratification in the character varieties $X_2 = X(\G_n, \U(2))$ and $\S X_3 = X(\G_n, \SU(3))$. As we will see, a much more involved geometry arises in these cases, with many strata interacting with non-trivial intersection patterns.

\subsection{$\U(2)$-character variety}\label{sec:u2-character-variety}

As in Section \ref{sec:SU2character}, we get again two possible cases arising from the partitions $\pi_0=(2)$ and $\pi_1=(1,1)$. 
\begin{enumerate}
\item $\pi_0 = (2)$ corresponds to irreducible representations $X_2^*$ and, as for $\SU(2)$, the only configuration of eigenvalues that contributes is $\sigma = (1,1)$ (which has parameters $r_1 = 1$ and $a_1=2$ in the notation of Section \ref{sec:irreducible}). Hence, the number of different eigenvalues is $N = a_1 = 2$ and thus by Corollary \ref{cor:irreducible-u-simplecase}, we have that
$$
	X_2^* = X_{(1,1)}^* = (S^1 \times \bm{\mu}_n^* \times  \mathbf{F}_{(1,1)})/\ZZ_2,
$$
where we have used that $ \Delta^{1}_{\bm{\mu}_n} = \bm{\mu}_n^*$. The action of $\ZZ_2$ on $S^1 \times \bm{\mu}_n^* \times {\mathbf{F}}_{(1,1)}$ is $(\lambda, \epsilon, B) \mapsto (\lambda \epsilon, \epsilon^{-1}, P_0BP_0^{-1})$, where $P_0$ is the permutation matrix that exchanges the columns of $B$. We have two options:

\begin{enumerate}
	\item If $n$ is odd, then there exists a unique representative $(\lambda, \epsilon, B)$ with $\Im(\epsilon) > 0$. Hence, using that $ {\mathbf{F}}_{(1,1)} = \S{\mathbf{F}}_{(1,1)} \times S^1 =\mathbf{D} \times S^1$ with $\mathbf{D}$ an open $2$-dimensional disc (c.f.\ Section \ref{sec:SU2character}), we get
$$
	X_2^* = S^1 \times \bm{\mu}_n^+ \times \mathbf{D} \times S^1,
$$
where $\bm{\mu}_n^+ = \{\epsilon \in \bm{\mu}_n \,|\, \Im(\epsilon) > 0\}$. These are $(n-1)/2$ copies of $S^1 \times \mathbf{D} \times S^1$.

	\item If $n$ is even, we still get $(n-2)/2$ copies of $S^1 \times \mathbf{D} \times S^1$, corresponding to the roots of unity with $\Im(\epsilon) > 0$. However, for $\epsilon = -1$ we get a residual action of $\ZZ_2$ on $S^1 \times \mathbf{D} \times S^1$. Since the action in the first copy of $S^1$ is $\lambda = e^{2\pi i \theta} \mapsto -\lambda = e^{2\pi i (\theta + 1/2)}$, there is an unique representative with $0 \leq \theta < 1/2$. Hence, in this case we get a contribution $\mathbf{I}_0 \times \mathbf{D} \times S^1$, where $\mathbf{I}_0 = [0,1/2)$.
\end{enumerate}

\item $\pi_1=(1,1)$ corresponds to the set of totally reducible representations $X^{\textrm{TR}}_2 = \Sym^2(S^1 \times S^1)$. As proven in Section \ref{sec:sym-2}, we have a $S^2$-fibration
$$
	S^2 \to \Sym^2(S^1 \times S^1) \to S^1 \times S^1.
$$
The monodromy of this fibration is
$(\lambda,\mu) \mapsto (e^{i\pi}\lambda, e^{i\pi}\mu)=(-\lambda,-\mu)$, that is an orientation preserving action interchanging 
the four orbifold points of $S^2$ in pairs.
\end{enumerate}

\subsection{$\SU(3)$-character variety}

In this case, we have three possible semisimple types corresponding to the partitions $\pi_0=(3)$, $\pi_1=(1,1,1)$ and $\pi_2 = (1,2)$. 

\begin{itemize}
	\item For $\pi_0=(3)$ we get irreducible representations $\S X_3^*$. We stratify according to the repeated eigenvalues of the matrix $A$ of a representation $(A,B) \in \S X^*_3$ as in (\ref{eq:stratification-eigen-rep}) by
	$$
		\S X_3^* = \S X_{(1,1,1)}^* \sqcup \S X_{(1,2)}^* \sqcup \S X_{(3)}^*.  
	$$
	\begin{enumerate}
		\item $\sigma = (1,1,1)$ (three different eigenvalues). By the results of Section \ref{sec:distinct-eigenvalues}, we have that the space $\S X_{(1,1,1)}^*$ is a collection of
		$$
			N_{(1,1,1)} = \frac{3}{n} \begin{pmatrix}n \\ 3\end{pmatrix} = \frac{(n-1)(n-2)}{2}
		$$
		copies of a certain subspace $\S \mathbf{F}_{(1,1,1)} \subset \SU(3)/(S^1)^3$.
		
		For the projection map onto the coarse orthant $\varphi: \SU(3)/(S^1)^3 \to B_3 = \{(z, x_2, x_3) \in \CC \times \RR_{\geq 0}^2\,|\, |z|^2 + x_2^2 + x_3^2 = 1\} \subset S^3$, on the interior of $B_3$ the preimage is exactly $\S \mathbf{F}_{(1,1,1)}$ and the fiber is $S^3$. On an edge of $\partial B_3$, let us say the one in which $x_2 = 0$, fixed a first column $(z, 0, x_3) \in \partial B_3$, the fiber is isomorphic to the quotient under the action of $S^1 \times S^1$ of the space
		\begin{alignat}{3}
			\qquad & \qquad & \left\{(w_1, w_2, w_3) \in \CC^3 \,\left|\, w_3 = - \frac{w_1\bar{z}}{x_3}, \,\, |w_1|^2 + |w_2|^2 + |w_3|^2 =1, \,\, (A,B)\textrm{ is irreducible}\right.\right\}.\nonumber
		\end{alignat}
		Since, $(A,B)$ must be irreducible, we have that $w_1, w_3 \neq 0$ and hence, using the action of $S^1 \times S^1$ we find an unique point in the orbit with $w_1 \in \RR_{>0}$. Hence, the fiber of $\S \mathbf{F}_{(1,1,1)}$ on an edge is the orthant $B_2 \subset S^3$, which is a $2$-dimensional disc. Finally, the vertices of the orthant are not included since otherwise the representation would be reducible.
		\item $\sigma = (1,2)$ (two coincident eigenvalues). In this case, the space $\S X_{(1,2)}^*$ is a collection of
		$$
			N_{(1,2)} = \frac{3}{n} \begin{pmatrix}n \\ 2\,1\end{pmatrix} = \frac{3(n-1)(n-2)}{2}
		$$
		copies of a certain subspace $\mathbf{F}_{(1,2)} \subset \SU(3)/(\U(2) \times S^1)$, where $\U(2) \times S^1$ is the stabilizer of the type $\sigma = (2,1)$. Using the action of this stabilizer, any representation $(A,B)$ can be put in the form $A = \textrm{diag}(\lambda_1, \lambda_1, \lambda_2)$ and
		$$
			B = \begin{pmatrix} x_1 & y_1 & z_1 \\ 0 & y_2 & z_2 \\ x_3 & y_3 & z_3 \end{pmatrix}.
		$$ 
		The stabilizer of this shape of matrices is now $(S^1)^3$, with action as in Section \ref{sec:distinct-eigenvalues}. This corresponds exactly to an edge in case (1), so we get that $\mathbf{F}_{(1,2)} \cong \mathbf{I} \times B_2$.
		\item $\sigma = (3)$ (three equal eigenvalues). There are no elements in this stratum since they would be reducible.
	\end{enumerate}
	
	\item For $\pi_1=(1,1,1)$ we get totally reducible representations, which are $\S X_3^{\mathrm{TR}} = \SSym^3(S^1 \times S^1)$. This space is described in Section \ref{sec:sym-3}.
	\item For $\pi_2=(1,2)$, we get that $\S X^{(1,2)}_3 = X_2$. This space was described in Section \ref{sec:u2-character-variety}.
\end{itemize}

Despite this description provides an accurate geometric picture, the intersection pattern it describes is too involved to compute the homotopy types of these character varieties. A prospective future work is to seek  alternative approaches that fully characterize these intersections, enabling effective homological and homotopical calculations in this higher rank case.

\end{document}